\documentclass[11pt]{article}

\usepackage{graphicx}
\usepackage{amsthm,amsmath,amssymb,tikz,bm}
\usepackage{mathtools}
\usepackage{xifthen}
\usepackage{listings}
\usepackage{comment}
\usepackage{thmtools}
\usepackage[margin=1.0in]{geometry}
\usepackage[shortlabels]{enumitem}
\usepackage{todonotes}
\usepackage[colorlinks=true,
linkcolor=blue,citecolor=blue,
urlcolor=blue]{hyperref}

\usetikzlibrary{calc,shapes,backgrounds,positioning}

\tikzset{
    comp/.style={
        draw,
        rounded corners=2pt,
        minimum width=9mm,
        minimum height=18mm,
        align=center
    },
    smallcomp/.style={
        draw,
        minimum width=10mm,
        minimum height=8mm,
        align=center
    },
    blob/.style={
        draw,
        ellipse,
        minimum width=5.8cm,
        minimum height=2.2cm
    },
    region/.style={
        draw,
        ellipse,
        minimum width=11mm,
        minimum height=22mm,
        align=center
    },
    vtx/.style={
        circle,
        fill,
        inner sep=1.3pt
    }
}

\newtheorem{theorem}{Theorem}[section]
\newtheorem{lemma}[theorem]{Lemma}
\newtheorem{corollary}[theorem]{Corollary}

\newtheorem{claim}{Claim}

\newtheorem{conjecture}{Conjecture}
\newtheorem{proposition}[theorem]{Proposition}



\makeatletter
               {\list{}{\leftmargin=0pt
                        \labelwidth\z@ \itemindent-\leftmargin
                        }}%
               {\endlist}
\makeatother

\title{4-Arc-Pancyclicity of Regular Multipartite Tournaments}
\author{
Weihao Xia \thanks{Department of Mathematics, Louisiana State University,
Baton Rouge, LA 70803 (\texttt{wxia3@lsu.edu}).}
}
\date{}

\begin{document}

\maketitle
\begin{abstract}
A multipartite tournament is an orientation of a complete multipartite graph.
We prove that every $r$-regular $c$-partite tournament with common partite-set cardinality $\alpha$ is $4$-arc-pancyclic whenever $c\ge93$; that is, every arc belongs to a cycle of each length from $4$ to $c\alpha$.
 This confirms the conjecture of Zhou and Zhang for all sufficiently large $c$ and provides a multipartite analog of Alspach's arc-pancyclicity theorem. Moreover, we also give a construction to show that 4-arc-pancyclic is the best possible.
Next, we prove that every arc belongs to at least $c\alpha-\alpha-1$ cycles of pairwise distinct lengths when $c\ge7$ and $\alpha\ge2$.
For regular $3$-partite tournaments with common partite-set cardinality $\alpha\ge2$, we obtain the sharp lower bound $\alpha$, settling the remaining case of a conjecture of Xia, Cai, Guo, and Wang.
\end{abstract}

\section{Introduction}

A multipartite tournament is an orientation of a complete multipartite graph, and a tournament is the special case in which every partite set is a singleton.
A digraph is \emph{$r$-regular} if $d^+(v)=d^-(v)=r$ for every vertex $v$. For a multipartite tournament,
regularity forces all partite sets to have the  same cardinality, which will be denoted by $\alpha$.
For an $r$-regular $c$-partite tournament with partite sets of cardinality $\alpha$, we have $r=\frac{(c-1)\alpha}{2}$.

An arc $e$ (or a vertex $v$) of a digraph $D$ is called \emph{$k$-pancyclic} if $e$ (or $v$) belongs to a cycle of length $\ell$ for every integer $\ell$ satisfying
$k\le \ell\le |V(D)|$.
The digraph $D$ is \emph{$k$-arc-pancyclic (or $k$-vertex-pancyclic)} if every arc (or vertex) of $D$ is $k$-pancyclic. When $k=3$, we simply say that $D$ is \emph{arc-pancyclic} or \emph{vertex-pancyclic}, respectively.

Alspach \cite{Alspach1967} proved that every regular tournament is arc-pancyclic.
Regular multipartite tournaments may contain an arc belonging to no triangle, so arc-pancyclicity cannot be guaranteed in general.
This leads to the question of whether every arc belongs to a cycle of each length from $4$ to the order of the $c$-partite tournament. 
Zhou and Zhang \cite{ZhouZhang2002} proved that every arc of a regular $c$-partite tournament with $c\ge6$ belongs to a cycle of each length from $4$ to $c$. 
They also constructed a regular $5$-partite tournament containing an arc that does not belong to a $4$-cycle.
Hence, they proposed the following conjecture.
\begin{conjecture}[\cite{ZhouZhang2002}]\label{conj:four-arc-pancyclic}
Let $c\ge6$, and let $T$ be an $r$-regular
$c$-partite tournament in which every partite set has cardinality
$\alpha$.
Then  $T$ is 4-arc-pancyclic. 
\end{conjecture}

Our main result confirms Conjecture~\ref{conj:four-arc-pancyclic} when $c\ge93$.
It provides a multipartite analog of Alspach's arc-pancyclicity theorem.
For $\alpha\ge2$, Proposition~\ref{prop:noC3} shows that the starting length $4$ cannot be replaced by $3$.

\begin{theorem}\label{thm:asymptotic-n-minus-three}
Let $c\ge93$, and let $T$ be an $r$-regular $c$-partite tournament in which every partite set has cardinality $\alpha$.
Then $T$ is $4$-arc-pancyclic.
\end{theorem}

Lemma~\ref{lem:T-Z} shows that every arc lies on a Hamiltonian cycle after deleting any set of at most $\alpha$ vertices avoiding its endpoints.
These cycles provide the longest cycle lengths needed in the proof of Theorem~\ref{thm:asymptotic-n-minus-three}.

The following theorem supplies the shorter cycle lengths needed in the proof of Theorem \ref{thm:asymptotic-n-minus-three} and holds under the weaker assumption $c\ge7$.

\begin{theorem}\label{thm:alphac-alpha-1}
Let $c\ge7$ and $\alpha\ge2$, and let $T$ be an $r$-regular $c$-partite tournament in which every partite set has cardinality $\alpha$.
Then every arc of $T$ belongs to at least $c\alpha-\alpha-1$ cycles of pairwise distinct lengths.
Moreover, the following statements hold for every arc $xy$.
\begin{enumerate}[(1)]
\item If $xy\in C_3$, then $xy$ belongs to a cycle of length $\ell$ for every $\ell\in\{3,4,\ldots,c\alpha-\alpha+1\}$.
\item If $xy\notin C_3$, then  $xy$ belongs to a cycle of length $\ell$ for every $\ell\in\{4,5,\ldots,c\alpha-\alpha+2\}$.
\end{enumerate}
\end{theorem}

Since $|V(T)|=c\alpha$, Theorem~\ref{thm:asymptotic-n-minus-three} immediately yields the following corollary.
\begin{corollary}
Let $c\ge93$, and let $T$ be an $r$-regular $c$-partite tournament in which every partite set has cardinality $\alpha$.
Then every arc of $T$ belongs to at least $c\alpha-3$ cycles of pairwise distinct lengths.
\end{corollary}

For $\alpha=2$, Theorem~\ref{thm:alphac-alpha-1}, together with the computation for $c\in\{4,5,6\}$ given in the Appendix, yields the following corollary. 

\begin{corollary}\label{cor:alpha-two}
Let $c\ge4$, and let $T$ be a $(c-1)$-regular $c$-partite tournament in which every partite set has cardinality two.
Then every arc of $T$ belongs to at least $2c-3$ cycles of pairwise distinct lengths.
\end{corollary}
\begin{proof}
Suppose that $c\ge7$, then by Theorem~\ref{thm:alphac-alpha-1}, we are done.
For $c\in\{4,5,6\}$, the assertion was verified by the exhaustive SageMath computation given in the Appendix.
\end{proof}

The behavior of regular multipartite tournaments differs between the cases $c=3$ and $c\ge4$.
Volkmann \cite{Volkmann2006} showed that, for every fixed integer $t\geq3$, there exist infinitely many regular $3$-partite tournaments containing an arc that belongs to no  $t$-cycle. Thus, in the tripartite case, no prescribed actual cycle length can be guaranteed uniformly for every arc. Cycles through a given arc with lengths in a prescribed short interval have been studied in \cite{GuoCuiMeng2016,GuoMeng2019,Volkmann2006}.

Recently, Xia, Cai, Guo, and Wang  \cite{XiaCaiGuoWang2026} proved that, for $c\ge 3$, every arc of a regular $c$-partite tournament lies on at least $c-2$ cycles of pairwise distinct lengths.
They proposed the following conjecture for $r$-regular multipartite tournaments.

\begin{conjecture}[\cite{XiaCaiGuoWang2026}]\label{conj:c+a-3}
Let $T$ be an $r$-regular $c$-partite tournament with partite sets of cardinality $\alpha$, where $c\ge3$.
Then every arc of $T$ belongs to at least $c+\alpha-3$ cycles of pairwise distinct lengths.
\end{conjecture}

For $\alpha=1$, Conjecture~\ref{conj:c+a-3} follows from Alspach's theorem.
For $c\ge4$ and $\alpha\ge2$, Theorem~\ref{thm:vypath} and Lemma~\ref{lem:r-cycle} give at least $r=(c-1)\alpha/2$ distinct cycle lengths through every arc.
Thus only the case $c=3$ and $\alpha\ge2$ remains.
Our next result settles this case.

\begin{theorem}\label{thm:c=3}
Let $\alpha\ge2$, and let $T$ be an $r$-regular $3$-partite tournament in which every partite set has cardinality $\alpha$.
Then every arc of $T$ belongs to at least $\alpha$ cycles of pairwise distinct lengths.
Moreover, this bound is tight.
\end{theorem}

The proof of Theorem \ref{thm:c=3} relies on Lemma~\ref{lem:hp}, which shows that every arc of an $r$-regular $3$-partite tournament with partite sets of cardinality $\alpha\ge2$ is the initial arc of a Hamiltonian path.

The contrast between Theorem~\ref{thm:c=3} and
Theorem~\ref{thm:asymptotic-n-minus-three} suggests that the case
$c\geq4$ should behave fundamentally differently from the tripartite
case. We therefore propose the following conjecture.

\begin{conjecture}\label{conj:n-3}
Let $c\ge4$ and $\alpha\ge2$, and let $T$ be an $r$-regular
$c$-partite tournament in which every partite set has cardinality
$\alpha$.
Then every arc of $T$ belongs to at least $c\alpha-3$ cycles of pairwise distinct lengths.
\end{conjecture}

For related results on vertex-pancyclicity, see \cite{GuoVolkmann1994,Moon1966,Yeo1999Vertex}.
Other multipartite analogues of Alspach's theorem, formulated in terms of the partite sets met by paths and cycles, can be found in \cite{CXY2023,GuoKwak2002,PanZhang2004}.
For surveys of paths and cycles in multipartite tournaments, see \cite{Volkmann2002,Volkmann2007}.

\medskip

Section~\ref{sec:prescribed-paths} proves the tripartite result, Theorem~\ref{thm:c=3}.
Section~\ref{sec:alpha-two} proves Theorem~\ref{thm:alphac-alpha-1}.
Section~\ref{sec:c>200} proves our main result, Theorem \ref{thm:asymptotic-n-minus-three},  and gives a construction showing that $4$-arc-pancyclic is the best possible.

\medskip
\noindent
\textbf{Notation.}
Unless stated otherwise, all paths and cycles in this paper are directed.
For a vertex $v$ on a cycle, let $v^+$ and $v^-$ denote its successor and predecessor, respectively.
If $a$ and $b$ are on a path or cycle $C$, let $C[a,b]$ denote the segment from $a$ to $b$, including both ends.
For disjoint vertex sets $A$ and $B$, let $E(A,B)$ be the set of arcs $ab$ with $a\in A$ and $b\in B$, and let $\varepsilon(A,B)=|E(A,B)|+|E(B,A)|$.
If $v\notin V(C)$, let $N_C^+(v)=\{u\in V(C):v\to u\}$,
$N_C^-(v)=\{u\in V(C):u\to v\}$, and $N_C(v)=N_C^+(v)\cup N_C^-(v)$.
Let $C$ and $Q$ be vertex-disjoint cycles.
A vertex $v\in V(Q)$ is \emph{out-singular} with respect to $C$ if
$N_C(v)=N_C^+(v)$, and it is \emph{in-singular} with respect to $C$ if
$N_C(v)=N_C^-(v)$.
A vertex is \emph{singular} with respect to $C$ if it is either
out-singular or in-singular with respect to $C$.
A cycle factor of a digraph $D$ is a collection of pairwise vertex-disjoint cycles covering $V(D)$.

\section{Proof of Theorem \ref{thm:c=3}}\label{sec:prescribed-paths}

To prove Theorem~\ref{thm:c=3}, we need the following lemmas. 

\begin{lemma}\label{lem:r-cycle}
Let $D$ be an $r$-regular oriented graph.
If  $xy$ is the initial arc of a Hamiltonian path of $D$, then $xy$ belongs to at least $r$  cycles of pairwise distinct lengths.
\end{lemma}

\begin{proof}
Let  $xyx_1x_2\cdots x_{n-2}$ be the  Hamiltonian path where $n=|V(D)|$.
Since $D$ is $r$-regular, every vertex of $D$ has exactly $r$ in-neighbors. 
There are exactly $r$ indices $i\in\{1,\ldots,n-2\}$ such that $x_i\to x$.
For each $i$, $xyx_1\cdots x_ix$ is a cycle containing $xy$ with length $i+2\ge 3$.
 Thus, $xy$ belongs to at least $r$  cycles of pairwise distinct lengths.
\end{proof}

\begin{theorem}[\cite{BangJensenWangYeo2025,Yeo1997}]\label{thm:yeo}
Let $F=C_1\cup C_2\cup\cdots\cup C_t$ be a cycle factor of a multipartite tournament $D$, where $t$ is minimum among all cycle factors of $D$.
Then the cycles of $F$ can be indexed so that the following statements hold. For every $i<j$,
\begin{enumerate}[(1)]
\item $C_i$ has singular vertices with respect to $C_j$, all of which are out-singular, and $C_j$ has singular vertices with respect to $C_i$, all of which are in-singular.
\item if $E(C_j,C_i)\ne\varnothing$, there is a partite set $P_{ij}$ such that every arc $y\to u$ from $C_j$ to $C_i$ satisfies $u^-,y^+\in P_{ij}$, and $u\to y^+$ and $u^-\to y$ are arcs.
\end{enumerate}
\end{theorem}

\begin{lemma}\label{lem:cut}
Let $D$ be a multipartite tournament in which every partite set has cardinality at most $\alpha$, and let $A,B\subseteq V(D)$ be disjoint and nonempty.
Then
\[
\varepsilon(A,B)\left(\frac{1}{|A|}+\frac{1}{|B|}\right)\ge |A|+|B|-\alpha.
\]
\end{lemma}

\begin{proof}
Let $x=|A|$ and $y=|B|$.
For each partite set $V_i$, let $x_i=|A\cap V_i|,$  $y_i=|B\cap V_i|,$ and $ s_i=x_i+y_i.$
There are $\sum_i x_iy_i$ nonedges between $A$ and $B$.

For each $i$ with $s_i>0$, let $x_i=s_i\theta_i$ and $y_i=s_i(1-\theta_i)$.
Since $s_i\le\alpha$, the concavity of $t(1-t)$ gives
\[
\sum_i x_iy_i
\le \alpha\sum_i s_i\theta_i(1-\theta_i)
\le \alpha(x+y)\frac{x}{x+y}\frac{y}{x+y}
=\frac{\alpha xy}{x+y}.
\]
Hence,
\[
\varepsilon(A,B)=xy-\sum_i x_iy_i
\ge \frac{xy(x+y-\alpha)}{x+y}.\qedhere
\]
\end{proof}

For a digraph $D$, its global irregularity is defined by
\[
i_g(D)=\max_{v\in V(D)}\{d^+(v),d^-(v)\}-\min_{v\in V(D)}\{d^+(v),d^-(v)\}.
\]

\begin{theorem}[\cite{VolkmannYeo2004}]
\label{thm:vypath}
Let $D$ be a $c$-partite tournament with partite sets $V_1,V_2,\ldots,V_c$, where
$|V_1|\le |V_2|\le\cdots\le |V_c|$.
Let $P$ be a path of length $p$ in $D$.
If
\[
i_g(D)\le
\frac{|V(D)|-3p-2|V_c|-|V_{c-1}|+2}{2}
\]
then $D$ has a Hamiltonian path whose initial segment is $P$.
\end{theorem}

Let $c\ge4$ and $\alpha\ge2$, and let $T$ be an $r$-regular $c$-partite tournament in which every partite set has cardinality $\alpha$.
Let $xy$ be an arbitrary arc of $T$.
For $D=T$ and $P=xy$, the condition in Theorem~\ref{thm:vypath} becomes $c\alpha\ge3\alpha+1$, which holds since $(c-3)\alpha\ge2$.
Thus, $T$ has a Hamiltonian path whose initial arc is $xy$.
By Lemma~\ref{lem:r-cycle}, $xy$ belongs to at least $r$ cycles of pairwise distinct lengths.
Since $r-(c+\alpha-3)=(c-3)(\alpha-2)/2\ge0$, Conjecture~\ref{conj:c+a-3} holds for $c\ge4$ and $\alpha\ge2$.
It remains to consider the case $c=3$.

\begin{lemma}\label{lem:hp}
Let $\alpha\ge2$, and let $T$ be an $r$-regular $3$-partite tournament in which every partite set has cardinality $\alpha$.
Then every arc of $T$ is the initial arc of a Hamiltonian path.
\end{lemma}

\begin{proof}
Fix an arc $xy$ of $T$.
Let $T_1=\{u^1:u\in V(T)\}$ and $T_2=\{u^2:u\in V(T)\}$ be two disjoint copies of $V(T)$.
Let $\mathcal B=(T_1,T_2)$ be the bipartite graph in which $u^1v^2\in E(\mathcal B)$ if and only if $u\to v$ in $T$.
Since $T$ is $\alpha$-regular, $\mathcal B$ is also $\alpha$-regular.

Since $\mathcal B$ is $\alpha$-regular and bipartite, by Hall's Theorem, there is a perfect matching $M_1$. After deleting $M_1$, $\mathcal B-M_1$ is $(\alpha-1)$-regular and bipartite. Repeating this argument, we obtain pairwise edge-disjoint perfect matchings $M_1,\ldots,M_\alpha$ of $\mathcal{B}$.  Therefore,
\[
E(\mathcal B)=M_1\sqcup\cdots\sqcup M_\alpha.
\]
Note that a perfect matching in $\mathcal B$ corresponds to a cycle factor of $T$.

Let $M_i$ be the perfect matching containing $x^1y^2$ in $\mathcal{B}$.
The edges of $M_i$ corresponding to arcs of $T$ form a cycle factor containing $xy$.

In $T$, delete $x$ and add a new vertex $w$ such that $\{w\}$ is a singleton partite set. Let $w\to y$ and $v\to w$ for every $v\in V(T)\setminus\{x,y\}.$
Let $D$ be the resulting multipartite tournament.

Consider the cycle factor that contains $xy$.
Replace $x$ on this cycle with $w$.
The resulting cycles also form a cycle factor of $D$.

Let $F=C_1\cup C_2\cup\cdots\cup C_t$ be a cycle factor of $D$ with the minimum possible number of cycles, labeled as in Theorem~\ref{thm:yeo}.
Let $C=C_k$ be the cycle of $F$ containing $w$.
Since $y$ is the only out-neighbor of $w$, we have $wy\in E(C)$.

If $t=1$, then $F=C_1$ and $C_1$ is a Hamiltonian cycle of $D$, we are done. Suppose that $t\ge 2$. 
Note that every vertex outside $C_k$ dominates $w$.
Thus, $w$ is in-singular with respect to every other cycle of $F$.

We claim that $k=t$.
Suppose that $k<t$.
By Theorem~\ref{thm:yeo}(1), every singular vertex of $C_k$ with respect to $C_{k+1}$ is out-singular.
This contradicts the fact that $w$ is in-singular with respect to $C_{k+1}$.
Therefore, $k=t$, and hence, $C=C_t$.

Let $R=V(D)\setminus V(C)$,  $s=|R|$, and $m=|C|$. Then $m+s=3\alpha$.
Since $C$ is a cycle and $R$ has at least one cycle,  $3\le s,m\le3\alpha-3$.
Let $\lambda=|E(V(C),R)|$ and $\mu=|E(R,V(C))|$.

\begin{claim}\label{cl:mu}
 $\mu\ge s+2\lambda$.
\end{claim}

\begin{proof}
Consider an arc $e=z\to u$ from $C$ to $R$, and let $C_i$ be the cycle of $F$ containing $u$.
By Theorem~\ref{thm:yeo}(2), both $u\to z^+$ and $u^-\to z$ belong to $E(R,V(C))$. For each
$e=z\to u\in E(V(C),R)$, define $f(e)=u\to z^+$ and $g(e)=u^-\to z.$
Clearly, $f$ and $g$ are injective. 

We claim that arcs in
\[
\{f(e):e\in E(V(C),R)\}
\cup
\{g(e):e\in E(V(C),R)\}
\]
are pairwise distinct.

Since $f$ and $g$ are injective, it remains to show that the images of $f$ and $g$ are
disjoint. Suppose not. Then there are arcs
$z\to u$ and $z'\to u'$ from $C$ to $R$ such that $f(z\to u)=g(z'\to u'),$ thus $u\to z^+=u'^-\to z'.$
Therefore, $u'=u^+$ and $z'=z^+$.
Note that $u$ and $u'$ belong to the same cycle $C_i$.
Applying Theorem~\ref{thm:yeo}(2) to $z\to u$ and $z^+\to u^+$ gives $u^-,z^+\in P_{it}$ and $u,z^{++}\in P_{it}$, respectively.
This is impossible because $z^+$ and $z^{++}$ are consecutive vertices of $C$ and cannot belong to the same partite set.
Therefore, $f$ and $g$ have disjoint images, and hence, $\mu\ge2\lambda$.

Next, we want to show that none of these arcs has head $w$. Indeed, an arc $g(z\to u)=u^-\to z$ could have head $w$ only if $z=w$, but $w$ has no out-neighbor in $R$. Moreover, if $f(z\to u)=u\to z^+$ had head $w$, then $z^+=w$. By Theorem~\ref{thm:yeo}(2), $u^-$ and $z^+=w$ belong to the same partite set. Since $\{w\}$ is a singleton partite set, this implies that $u^-=w$, contrary to $u^-\in R$.

Since every vertex of $R$ dominates $w$, there are $s$ additional arcs from $R$ to $C$. Hence  $\mu\ge s+2\lambda.$
\end{proof}

Let $V_x$ be the partite set of $T$ containing $x$, and let $U_0=N_T^-(x)$, $U_1=V_x\setminus\{x\}$, and  $U_2=N_T^+(x)\setminus\{y\}$ where $U_0,U_1,U_2$ have cardinalities $\alpha,\alpha-1,\alpha-1$, respectively.
By the construction of $D$, we have 
\[
d_D^+(v)-d_D^-(v)=
\begin{cases}
0,&v\in U_0\cup\{y\},\\
1,&v\in U_1,\\
2,&v\in U_2,\\
-3(\alpha-1),&v=w.
\end{cases}
\]
For $i\in\{0,1,2\}$, let $\rho_i=|R\cap U_i|$, and let
\[
\sigma=\sum_{v\in R}\bigl(d_D^+(v)-d_D^-(v)\bigr)=2\rho_2+\rho_1.
\]
Since $w,y\in V(C)$, we have $s=\rho_0+\rho_1+\rho_2$.
Moreover, only arcs between $R$ and $C$ contribute to $\sigma$, so $\sigma=\mu-\lambda.$
By Claim~\ref{cl:mu}, $\sigma=\mu-\lambda\ge s+\lambda,$ and hence, $\lambda\le\sigma-s$.
It follows that
\[
\varepsilon(V(C),R)
=\lambda+\mu
=2\lambda+\sigma
\le3\sigma-2s.
\]

\begin{claim}
$m=3$.
\end{claim}

\begin{proof}
Since $\sigma=s+\rho_2-\rho_0\le s+\alpha-1$ and $m+s=3\alpha$, we have
\[
\varepsilon(V(C),R)\le3\sigma-2s\le s+3\alpha-3=m+2s-3.
\]
By Lemma~\ref{lem:cut}, we have $\varepsilon(V(C),R)\left(\frac{1}{m}+\frac{1}{s}\right)\ge s+m-\alpha.$
This gives us  $$m+2s-3\ge \varepsilon(V(C),R)\ge
\frac{2\alpha}{3\alpha/(ms)}
=\frac{2ms}{3}.$$
Hence, $(m-3)(2s-3)\le0$.
Since $m,s\ge3$, we have $m=3$.
\end{proof}

Since $m=3$, $s=3\alpha-3$.
Thus, $\sigma=2\rho_2+\rho_1\le3\alpha-3=s$, and hence,
\[
2s=\frac{2ms}{3}\le\varepsilon(V(C),R)\le3\sigma-2s\le s,
\]
a contradiction.

Therefore, $t=1$, and hence $F=C$, so $C=wyx_2\cdots x_{3\alpha-1}w$ is a Hamiltonian cycle of $D$.
Replacing $w$ with $x$, this gives the Hamiltonian path
$xyx_2x_3\cdots x_{3\alpha-1}$
of $T$ with initial arc $xy$.
\end{proof}

\begin{proof}[Proof of Theorem~\ref{thm:c=3}]
Fix an arc $xy$ of $T$.
By Lemma~\ref{lem:hp}, every arc of $T$ is the initial arc of a Hamiltonian path.
Since $T$ is an $r$-regular $3$-partite tournament, $r=\alpha$.
Hence, by Lemma~\ref{lem:r-cycle} there are at least $\alpha$ distinct cycle lengths through $xy$.

Moreover, the construction given by Example 3.1 of~\cite{XiaCaiGuoWang2026} shows that this bound is tight.
\end{proof}

\section{Proof  of  Theorem \ref{thm:alphac-alpha-1}}
\label{sec:alpha-two}

\begin{theorem}[\cite{Yeo1999}]
\label{thm:yeo-hc}
Let $D$ be a $c$-partite tournament with partite sets
$V_1,V_2,\ldots,V_c$, where
$|V_1|\le |V_2|\le\cdots\le |V_c|$.
If
\[
i_g(D)\le
\frac{|V(D)|-2|V_c|-|V_{c-1}|+2}{2},
\]
then $D$ has a Hamiltonian cycle.
\end{theorem}

The following lemma gives cycles of consecutive lengths containing a prescribed path.

\begin{lemma}\label{lem:path-extension}
Let $P=p_1p_2\cdots p_t$ be a  path in a digraph $D$, and suppose that $D-V(P)$ has a Hamiltonian cycle.
If
\[
|N_D^-(p_1)\setminus V(P)|+|N_D^+(p_t)\setminus V(P)|\ge |V(D)\setminus V(P)|+1,
\]
then $P$ belongs to a cycle of length $\ell$ for every $\ell\in\{t+1,t+2,\ldots,|V(D)|\}$.
\end{lemma}

\begin{proof}
Let $C=v_0v_1\cdots v_{n-1}v_0$ be a Hamiltonian cycle of $D-V(P)$, where $n=|V(D)\setminus V(P)|$.
Let $A=\{i:v_i\to p_1\}$ and $B=\{i:p_t\to v_i\}.$
Then $|A|+|B|\ge n+1$. Fix $q\in\{0,1,\ldots,n-1\}$, 
Let $B+q=\{i+q:i\in B\}$. Then $|B|=|B+q|$. 
Since $|A\cap(B+q)|\ge |A|+|B+q|-n\ge1$, there is an index $i\in B$ such that $i+q\in A$, where the indices are taken modulo $n$.
Thus, $P\,C[v_i,v_{i+q}]p_1$
is a cycle of length $t+q+1$.
Since $q\in\{0,1,\ldots,n-1\}$, $P$ belongs to a cycle of length $\ell$ for every $\ell\in\{t+1,t+2,\ldots,|V(D)|\}$.
\end{proof}

\begin{lemma}[\cite{GuoKwak2002,XiaCaiGuoWang2026}]\label{lem:short-cycle}
Let $T$ be an $r$-regular $c$-partite tournament with $c\ge3$.
Then every arc of $T$ belongs to a $3$-cycle or a $4$-cycle.
\end{lemma}

\begin{proof}[Proof  of  Theorem \ref{thm:alphac-alpha-1}]
Let $xy$ be an arc of $T$, and let $n=c\alpha$.
Let $A=N_T^-(x)$ and $S=N_T^+(y)$. Then $|A|=|S|=r$.
\begin{claim}\label{cl:T-hc}
    Let $H$ be any induced subdigraph obtained from $T$ by deleting exactly $\alpha+1$ vertices. Then $H$ has a Hamiltonian cycle.
\end{claim}
\begin{proof}
Let $H$ be any induced subdigraph obtained from $T$ by deleting exactly $\alpha+1$ vertices.
Then $|V(H)|=(c-1)\alpha-1$ and $i_g(H)\le\alpha+1$.
Since $c\ge 7$ and $\alpha\ge2$, we have
\[
i_g(H) \le \alpha+1 \le\frac{(c-4)\alpha+1}{2}.
\]
Hence, by Theorem~\ref{thm:yeo-hc}, $H$ has a Hamiltonian cycle.
\end{proof}

\medskip
\noindent
{\bf Case One: if  $A\cap S\ne\varnothing$.} Then $xy$ belongs to a triangle. 
Let $p=|A\cap S|\ge 1$.
Since $2r=n-\alpha$, we have
\[
\left|V(T)\setminus\bigl(\{x,y\}\cup A\cup S\bigr)\right|
=n-2-(2r-p)
=\alpha+p-2
\ge\alpha-1.
\]
Choose $Z\subseteq V(T)\setminus\bigl(\{x,y\}\cup A\cup S\bigr)$ such that $|Z|=\alpha-1$, and let $D=T-Z$.
Let $H=D-\{x,y\}$, so $|H|=n-\alpha-1$. Then by Claim \ref{cl:T-hc}, $H$ has a Hamiltonian cycle.
Since $Z$ is disjoint from $A\cup S$, we have
\[
|N_D^-(x)\setminus\{x,y\}|+
|N_D^+(y)\setminus\{x,y\}|
=|A|+|S|
=2r
=n-\alpha
=|V(H)|+1.
\]
By Lemma~\ref{lem:path-extension}, take $P=xy$ in $D$.
It follows that $xy$ belongs to a cycle of length $\ell$ for every
$\ell\in
\{3,4,\ldots,n-\alpha+1\}.$ Hence, this gives $n-\alpha-1$ distinct cycle lengths.

\medskip
\noindent
{\bf Case Two: if  $A\cap S=\varnothing$.}
By Lemma~\ref{lem:short-cycle}, $xy$ belongs to a $4$-cycle $C_4=xywax$ where $w\in S$, $a\in A$, and $w\to a$.
Hence, $a\in A\cap N_T^+(w)$. Note that $\{x,y,w\}\cap(A\cup N_T^+(w))=\varnothing$. Then 
\[
\begin{aligned}
\left|V(T)\setminus
\bigl(\{x,y,w\}\cup A\cup N_T^+(w)\bigr)\right|
&=n-3-|A\cup N_T^+(w)|\\
&=n-3-(|A|+|N_T^+(w)|-|A\cap N_T^+(w)|)\\
&\ge n-3-2r+1\\
&=\alpha-2.
\end{aligned}
\]
Choose $Z\subseteq
V(T)\setminus
\bigl(\{x,y,w\}\cup A\cup N_T^+(w)\bigr)$ such that $|Z|=\alpha-2$, and let $D=T-Z$.
Let $H=D-\{x,y,w\}$, so $|H|=n-\alpha-1$. Then by Claim \ref{cl:T-hc}, $H$ has a Hamiltonian cycle.
By the choice of $Z$, we have 
\[
|N_D^-(x)\setminus \{x,y,w\}|
+|N_D^+(w)\setminus \{x,y,w\}|
=|A|+|N_T^+(w)|=2r
=n-\alpha=|V(H)|+1.\]
By Lemma~\ref{lem:path-extension}, take $P=xyw$ in $D$.
It follows that $xy$ belongs to a cycle of length $\ell$ for every $\ell \in \{4,5,\ldots,n-\alpha+2\}.$
Thus,  $xy$ again belongs to cycles of $n-\alpha-1$ pairwise distinct lengths.
\end{proof}

\section{Proof  of   Theorem \ref{thm:asymptotic-n-minus-three}}\label{sec:c>200}

\begin{lemma}\label{lem:path-coloring}
Let $D$ be a multipartite tournament in which every partite set has cardinality at most $\alpha$, and let $U\subseteq V(D)$.
Then $U$ can be partitioned into at most $\alpha$ paths.
\end{lemma}

\begin{proof}
Since every partite set has cardinality at most $\alpha$, choose a map $\phi:U\to\{1,\ldots,\alpha\}$ such that vertices in the same partite set receive distinct colors.
Then each color class contains at most one vertex from each partite set. Thus, each color class induces a tournament.

Every tournament has a Hamiltonian path.
Thus, each nonempty color class has a  Hamiltonian path. Since every partite set has cardinality at most $\alpha$, $U $  can be partitioned into at most $\alpha$ paths.
\end{proof}

We next show the main lemma in this section, which states that an arc lies on a Hamiltonian cycle after deleting any set of at most $\alpha$ vertices avoiding its endpoints.

\begin{lemma}\label{lem:T-Z}
Let $c\ge 93$ and $\alpha\ge2$, and let $T$ be an $r$-regular $c$-partite tournament in which every partite set has cardinality $\alpha$.
Let $xy$ be an arc of $T$, and let $Z\subseteq V(T)\setminus\{x,y\}$ satisfy $|Z|\le\alpha$.
Then $T-Z$ has a Hamiltonian cycle containing $xy$.
\end{lemma}

\begin{proof}
Let $n=c\alpha$ and $s=|Z|$.
Let $G=T-Z$ and $N=|V(G)|=n-s$.
Let $d_0=r-s$, where $r=(n-\alpha)/2$.
Then $d_G^+(v),d_G^-(v)\ge d_0$ for every $v\in V(G)$.

Let $A=N_G^-(x)$, let $U=V(G)\setminus(A\cup\{x,y\})$, and let $u=|U|$.
Since $u=n-s-|A|-2$, $r-s\le|A|\le r$, and $s\le\alpha$, we have $r-2\le u\le n-r-2$.
Since $r-2>0$, we have $U\ne\varnothing$.

Every vertex of $G[U]$ is nonadjacent to at most $\alpha-1$ other vertices, so $|E(G[U])|\ge u(u-\alpha)/2$.
Choose $z\in U$ with maximum outdegree in $G[U]$.
Then $d_{G[U]}^+(z)\ge(u-\alpha)/2$.

Let $B=V(G)\setminus\bigl(A\cup N_G^+(z)\cup\{x,y,z\}\bigr)$ and let $b=|B|$.
Since $B=U\setminus(\{z\}\cup N_G^+(z))$, we have
\[
b=u-1-d_{G[U]}^+(z)
\le\frac{u+\alpha-2}{2}
\le\frac n4+\frac{3\alpha}{4}-2.
\]

By Lemma~\ref{lem:path-coloring}, $B\cup\{z\}$ can be partitioned into pairwise vertex-disjoint paths $Q_1,\ldots,Q_t$, where $1\le t\le\alpha$.
Since $z$ has no out-neighbor in $B$, the path containing $z$ ends at $z$. Label this path $Q_t$, and let $Q_0=xy$.
For each $i\in\{0,1,\ldots,t-1\}$, let $u_i$ be the terminal vertex of $Q_i$, and let $v_i$ be the initial vertex of $Q_{i+1}$.

\begin{claim}\label{cl:short-covering-path}
There are paths $R_0,R_1,\ldots,R_{t-1}$ such that, for every $i\in\{0,1,\ldots,t-1\}$,
\begin{enumerate}[(a)]
\item $R_i$ is a path from $u_i$ to $v_i$;
\item $R_i$ has length at most six;
\item the internal vertices of $R_0,R_1,\ldots,R_{t-1}$ are pairwise disjoint and lie outside $B\cup\{x,y,z\}$.
\end{enumerate}
Moreover, there is a path $P$ from $x$ to $z$ beginning with $xy$ and containing every vertex of $B$, with $p=|V(P)|\le n/4+23\alpha/4+1$.
\end{claim}

\begin{proof}

Suppose that $R_0,R_1,\ldots,R_{i-1}$ have been constructed with the required properties.
Let
\[
F_i=\bigl((B\cup\{x,y,z\})\setminus\{u_i,v_i\}\bigr)
\cup\bigcup_{j=0}^{i-1}\bigl(V(R_j)\setminus\{u_j,v_j\}\bigr),
\]
where the second union is empty when $i=0$.
Let $f_i=|F_i|$ and $J_i=G-F_i$.
Since $u_i$ and $v_i$ are distinct vertices of $B\cup\{x,y,z\}$ and each previously constructed path has at most five internal vertices, we have
\[
f_i\le b+1+5i
\le b+1+5(t-1)
\le\frac n4+\frac{23\alpha}{4}-6.
\]
Let $m_i=|V(J_i)|=n-s-f_i$ and $d_i=r-s-f_i$.
Then every vertex of $J_i$ has indegree and outdegree at least $d_i$.
Since $s\le\alpha$ and $f_i\le n/4+23\alpha/4-6$, we have
\[
4d_i-m_i
=n-2\alpha-3s-3f_i
\ge\frac{(c-89)\alpha}{4}+18
>0,
\]
where the last inequality follows from $c\ge93$.

For $j\in\{0,1,2,3\}$, let
\[
X_j=\left\{v\in V(J_i):\text{$J_i$ contains a path from $u_i$ to $v$ of length at most $j$}\right\},
\]
and let $x_j=|X_j|$.
Note that $X_0=\{u_i\}$ and $X_1=\{u_i\}\cup N_{J_i}^+(u_i)$.
Thus, $x_1\ge d_i+1$.

We want to show that $x_3\ge2d_i+1$.
Suppose not.
Since $x_1\ge d_i+1$, we have $x_2\le x_3<2d_i+1<2(d_i+1)\le2x_1$.
Then $x_2-x_1<x_2/2$.
Note that every out-neighbor of a vertex in $X_2$ belongs to $X_3$, and
every out-neighbor of a vertex in $X_1$ belongs to $X_2$.
Therefore, there is no arc from $X_1$ to $X_3\setminus X_2$.

We now count the arcs $vw\in E(J_i)$ with $v\in X_2$.
Since every vertex of $J_i$ has outdegree at least $d_i$, there are at least $d_i x_2$ such arcs.
At most $\binom{x_2}{2}$ of them have both endpoints in $X_2$.
For every remaining arc $vw$, we have $v\in X_2\setminus X_1$ and $w\in X_3\setminus X_2$.
There are at most $(x_2-x_1)(x_3-x_2)$ such arcs. Therefore,
\[
d_i x_2\le\binom{x_2}{2}+(x_2-x_1)(x_3-x_2).
\]
Since $x_2-x_1<x_2/2$ and $x_3-x_2\ge0$, we have
\[
d_i x_2
\le\frac{x_2(x_2-1)}2+\frac{x_2(x_3-x_2)}2=\frac{x_2(x_3-1)}2
<d_i x_2,
\]
where the last inequality follows from $x_3<2d_i+1$, a contradiction.
Thus, $x_3\ge2d_i+1$.

Let
\[
Y_3=\left\{v\in V(J_i):\text{$J_i$ contains a path from $v$ to $v_i$ of length at most three}\right\}.
\]
Similarly, we have $y_3=|Y_3|\ge2d_i+1$.
Since $4d_i>m_i$, we have $x_3+y_3\ge4d_i+2>m_i$.
Thus, $X_3\cap Y_3\ne\varnothing$.

Choose $w\in X_3\cap Y_3$.
By the definitions of $X_3$ and $Y_3$, there is a path from $u_i$ to $w$ of length at most three and a path from $w$ to $v_i$ of length at most three.
Their union contains a path $R_i$ from $u_i$ to $v_i$ of length at most six.
By the definition of $F_i$, the internal vertices of $R_i$ lie outside $B\cup\{x,y,z\}$ and avoid the internal vertices of $R_0,\ldots,R_{i-1}$.
This completes the inductive construction.

Let $P=Q_0R_0Q_1\cdots R_{t-1}Q_t$.
This path begins with $xy$, ends at $z$, and contains every vertex of $B$.
Since $Q_0,\ldots,Q_t$ contain $b+3$ vertices in total and each of the $t$ connecting paths has at most five internal vertices, we have
\[
p\le b+3+5t
\le\frac n4+\frac{23\alpha}{4}+1.\qedhere
\]
\end{proof}

Let $P$ be the path given by Claim~\ref{cl:short-covering-path}, and let $W=V(G)\setminus V(P).$
Since $s\le\alpha$ and $p\le n/4+23\alpha/4+1$, we have $|W|=n-s-p\ge3(c-9)\alpha/4-1\ge2$.
Since $B\cup \{x,y,z\}\subseteq V(P)$, we have $W\subseteq A\cup N_G^+(z).$

Let $D$ be a multipartite tournament with vertex set $W\cup\{v_P\}$, where $\{v_P\}$ is a singleton partite set.
Let $D[W]=G[W]$.
For each $w\in W$, orient the edge incident with $v_P$ so that
\[
(A\cap W)\to v_P
\qquad\text{and}\qquad
v_P\to W\setminus A.
\]

\begin{claim}\label{d<r}
Every vertex of $D$ has indegree and outdegree at most $r$.
\end{claim}

\begin{proof}
Let $D_0=D-\{v_P\}=T[W]$.
If $w\in A\cap W$, then $w\to x$ in $T$ and $x\notin W$, so $d_{D_0}^+(w)\le r-1$.
Since $w\to v_P$ is the only arc added at $w$, we have $d_D^+(w)=d_{D_0}^+(w)+1\le r$ and $d_D^-(w)=d_{D_0}^-(w)\le r$.

If $w\in W\setminus A$, then $z\to w$ because $W\setminus A\subseteq N_G^+(z)$.
Since $z\notin W$, we have $d_{D_0}^-(w)\le r-1$.
The only arc added at $w$ is $v_P\to w$, so $d_D^-(w)=d_{D_0}^-(w)+1\le r$ and $d_D^+(w)=d_{D_0}^+(w)\le r$.

Finally, $d_D^-(v_P)=|A\cap W|\le|A|\le r$ and $d_D^+(v_P)=|W\setminus A|\le|N_G^+(z)|\le r$.
This proves the claim.
\end{proof}

Since $|V(D)|=n-s-p+1$ and every partite set of $D$ has cardinality at most $\alpha$, every vertex $v\in V(D)$ satisfies
\[
d_D^+(v)+d_D^-(v)\ge |V(D)|-\alpha=2r-s-p+1.
\]
By Claim~\ref{d<r}, for every vertex $v\in V(D)$, $r-s-p+1\le d_D^+(v),d_D^-(v)\le r$. 
Thus, $i_g(D)\le s+p-1\le s+p$.

To apply Theorem~\ref{thm:yeo-hc}, we need the following claim.
\begin{claim}
$\displaystyle i_g(D)\le s+p\le \frac{n-s-p-3\alpha+3}{2}.   $ 
\end{claim}

\begin{proof}
Since $i_g(D)\le s+p$, it remains to prove that $s+p\le(n-s-p-3\alpha+3)/2$.
This inequality is equivalent to $s+p+\alpha\le(n+3)/3$.
Since $s\le\alpha$ and $p\le n/4+23\alpha/4+1$, we have
\[
s+p+\alpha
\le\frac n4+\frac{31\alpha}{4}+1
\le\frac n3+1
=\frac{n+3}{3},
\]
where the second inequality follows from $n=c\alpha$ and $c\ge93$.
This proves the claim.
\end{proof}

Therefore, by  Theorem~\ref{thm:yeo-hc}, $D$ has a Hamiltonian cycle.
Let $C_D$ be a Hamiltonian cycle of $D$, and let $a\to v_P\to b$ on $C_D$.
Thus, $a,b\in W$. 
By the definition,  $a\in A$ and $b\notin A$.
Since $W\subseteq A\cup N_G^+(z)$,  $b\in N_G^+(z)$.
Therefore, $a\to x$ and $z\to b$.

Note that $P$ is a path from $x$ to $z$ beginning with $xy$ and $V(P)\cap W=\varnothing$.
Replacing the segment $av_Pb$ of $C_D$ with the path $aPb$ gives a Hamiltonian cycle of $G$ containing $xy$.
\end{proof}

Now, we combine Lemma~\ref{lem:T-Z} with Theorem~\ref{thm:alphac-alpha-1} to prove our main result.

\begin{proof}[Proof of Theorem~\ref{thm:asymptotic-n-minus-three}]
Let $n=c\alpha$, and let $xy$ be an arbitrary arc of $T$.
If $\alpha=1$, then by Alspach's Theorem, $T$ is 3-arc-pancyclic. Therefore, we assume that $\alpha\ge 2$.

For each $s\in\{0,1,\ldots,\alpha-2\}$, choose a set $Z_s\subseteq V(T)\setminus\{x,y\}$ with $|Z_s|=s$.
By Lemma~\ref{lem:T-Z}, $T-Z_s$ has a Hamiltonian cycle containing $xy$.
Therefore, $xy$ belongs to a  cycle of every length from $n-\alpha+2$ to $n$.

{\bf Case One: if  $xy$ belongs to a triangle.}
By Theorem~\ref{thm:alphac-alpha-1}, $xy$ belongs to a  cycle of every length from $3$ to $n-\alpha+1$.
Together with Lemma~\ref{lem:T-Z}, this shows that $xy$ belongs to a  cycle of every length from $3$ to $n$.

{\bf Case Two: if  $xy$ belongs to no triangle.}
By Theorem~\ref{thm:alphac-alpha-1}, $xy$ belongs to a  cycle of every length from $4$ to $n-\alpha+2$.
Together with Lemma~\ref{lem:T-Z}, this shows that $xy$ belongs to a  cycle of every length from $4$ to $n$.

Since $xy$ was arbitrary, every arc of $T$ belongs to a  cycle of every length from $4$ to $n$.
Hence, $T$ is $4$-arc-pancyclic.
\end{proof}

The following construction shows that $4$ cannot be replaced by $3$ in Theorem~\ref{thm:asymptotic-n-minus-three}.

\begin{proposition}\label{prop:noC3}
Let $c\ge3$ and $\alpha\ge2$, where $(c-1)\alpha$ is even.
Then there is an $r$-regular $c$-partite tournament with partite sets of cardinality $\alpha$ that contains an arc belonging to no  triangle.
\end{proposition}

\begin{proof}
Let $a=\lfloor\alpha/2\rfloor$ and $b=\lceil\alpha/2\rceil$.
Let $X_0,X_1,\ldots,X_{2c-1}$ be pairwise disjoint independent sets such that
\[
|X_i|=
\begin{cases}
a,&\text{if $i$ is even},\\
b,&\text{if $i$ is odd}.
\end{cases}
\]
All subscripts are taken modulo $2c$.

For each $i\in\{0,1,\ldots,c-1\}$, let $V_i=X_i\cup X_{i+c}.$
If $\alpha$ is even, then $a=b=\alpha/2$, and hence $|V_i|=\alpha$.
Suppose that $\alpha$ is odd.
Since $(c-1)\alpha$ is even, $c$ is odd.
Thus, $i$ and $i+c$ have opposite parity, so $|V_i|=a+b=\alpha$.
Therefore, $V_0,V_1,\ldots,V_{c-1}$ are partite sets of cardinality $\alpha$.

For every $i$, orient 
\[
X_i\to X_{i+1}\cup X_{i+2}\cup\cdots\cup X_{i+c-1}.
\]
There are no edges between $X_i$ and $X_{i+c}$ because these two sets belong to the same partite set.

Let $v\in X_i$, then
\[
d^+(v)=(c-1)\frac{\alpha}{2}\qquad \text{and} \qquad d^-(v)=\frac{(c-1)\alpha}{2}.
\]
Thus, the resulting $c$-partite tournament is $r$-regular, where $r=(c-1)\alpha/2$.

Choose $x\in X_0$ and $y\in X_1$.
By the construction, $x\to y$.
Suppose that $xy$ belongs to a triangle $xyzx$, where $z\in X_j$.
If   $y\to z$, then $j\in\{2,3,\ldots,c\}.$ However, if $z\to x$, then $j\in\{c+1,c+2,\ldots,2c-1\}.$
These two sets of indices are disjoint.
Thus, $xy$ belongs to no triangle.
\end{proof}

\section*{Acknowledgments}

The author thanks Guoli Ding, Yubao Guo, Xiaonan Liu, and Zhiyu Wang for helpful discussions and suggestions on this problem.

The SageMath code used in Corollary~\ref{cor:alpha-two} and the constructions in Proposition~\ref{prop:noC3} were generated using OpenAI's ChatGPT.
The author also used ChatGPT for language editing and to refine the proof of Theorem~\ref{thm:asymptotic-n-minus-three}, reducing the required lower bound on $c$ from $400$ to $93$.
All mathematical arguments, constructions, code, and computational results were independently checked by the author.

\section*{Appendix: Sage verification for Corollary~\ref{cor:alpha-two}}

\begin{lstlisting}[
language=Python,
basicstyle=\ttfamily\scriptsize,
breaklines=true,
columns=fullflexible,
frame=single,
showstringspaces=false
]
from itertools import combinations
from sage.sat.solvers.satsolver import SAT

def verify(c):
    n = 2 * c
    r = c - 1
    x, y = 0, 2
    V = range(n)
    part = {v: v // 2 for v in V}

    var = {}
    nxt = 1

    for u in V:
        for v in range(u + 1, n):
            if part[u] != part[v]:
                var[u, v] = nxt
                nxt += 1

    def arc(u, v):
        if u == v or part[u] == part[v]:
            return None
        if u < v:
            return var[u, v]
        return -var[v, u]

    solver = SAT(
        solver="cryptominisat",
        threads=4
    )

    # Every vertex has outdegree r.
    for u in V:
        outgoing = [
            arc(u, v) for v in V
            if part[u] != part[v]
        ]

        for X in combinations(outgoing, r + 1):
            solver.add_clause(
                tuple(-z for z in X)
            )

        for X in combinations(
            outgoing,
            len(outgoing) - r + 1
        ):
            solver.add_clause(X)

    # Fix the prescribed arc x -> y.
    solver.add_clause((arc(x, y),))

    # Symmetry breaking.
    remaining_parts = range(2, c)

    def pattern(p):
        return (
            arc(x, 2 * p),
            arc(x, 2 * p + 1),
            arc(y, 2 * p),
            arc(y, 2 * p + 1),
        )

    def mismatch(P, value):
        return [
            -z if (value >> (3 - i)) & 1 else z
            for i, z in enumerate(P)
        ]

    for p in remaining_parts:
        a, b, d, e = pattern(p)
        solver.add_clause((a, -b))
        solver.add_clause((a, b, d, -e))
        solver.add_clause((-a, -b, d, -e))

    for p, q in zip(
        remaining_parts,
        remaining_parts[1:]
    ):
        for a in range(16):
            for b in range(a + 1, 16):
                solver.add_clause(tuple(
                    mismatch(pattern(p), a)
                    + mismatch(pattern(q), b)
                ))

    # present[ell] is forced to be true
    # whenever there is an ell-cycle
    # containing the prescribed arc.
    present = {}

    for ell in range(3, n + 1):
        present[ell] = nxt
        nxt += 1

         # Search for at most 2*c-4 distinct cycle lengths.
    for X in combinations(tuple(present.values()), 2*c-3):
        solver.add_clause(tuple(-z for z in X))

    # Encode simple paths from y.
    internal = tuple(
        v for v in V
        if v not in (x, y)
    )
    base = 1 << y

    path = {(base, y): nxt}
    solver.add_clause((nxt,))
    nxt += 1

    for size in range(2, n):
        for X in combinations(
            internal, size - 1
        ):
            mask = base + sum(
                1 << v for v in X
            )

            for v in X:
                path[mask, v] = nxt
                nxt += 1

    for (mask, v), current in tuple(
        path.items()
    ):
        if v == y:
            continue

        old = mask - (1 << v)
        transitions = []

        for u in V:
            previous = path.get((old, u))
            edge = arc(u, v)

            if previous is None or edge is None:
                continue

            transition = nxt
            nxt += 1
            transitions.append(transition)

            solver.add_clause(
                (-transition, previous)
            )
            solver.add_clause(
                (-transition, edge)
            )
            solver.add_clause(
                (-previous, -edge, transition)
            )
            solver.add_clause(
                (-transition, current)
            )

        if transitions:
            solver.add_clause(tuple(
                [-current] + transitions
            ))
        else:
            solver.add_clause((-current,))

        closing = arc(v, x)

        if closing is not None:
            ell = int(mask).bit_count() + 1
            solver.add_clause(
                (-current,
                 -closing,
                 present[ell])
            )

    model = solver()

    if model is not False:
        raise AssertionError(
            "counterexample found for c = {}"
            .format(c)
        )

    print("c =", c, "result = UNSAT")

for c in (4, 5, 6):
    verify(c)

print("verification: PASSED")
\end{lstlisting}
The output is

\begin{lstlisting}[
basicstyle=\ttfamily\scriptsize,
frame=single,
showstringspaces=false
]
c = 4 result = UNSAT
c = 5 result = UNSAT
c = 6 result = UNSAT
verification: PASSED
\end{lstlisting}

\end{document}